\documentclass[12pt]{amsart}
\usepackage[initials]{amsrefs}
\usepackage{amssymb,amsmath,amsthm,amsfonts}
\usepackage[all]{xy}

\calclayout

\newtheorem*{thma}{Theorem A}
\newtheorem*{corb}{Corollary B}

\newcommand{\namedRef}[1]{%
\csname env:#1\endcsname\ \ref{#1}%
}
\newcommand{\namedType}[1]{\csname env:#1\endcsname}
\newtoks\displayedcounter\displayedcounter={dummy}

\def\envLRT#1#2{\expandafter\gdef\csname env:#1\endcsname{#2}}

\makeatletter
\newcommand{\namedLabel}[1]{\expandafter\xdef\csname env:#1\endcsname{\envName}%
\immediate\write\@auxout{\string\envLRT{#1}{\envName}}%
\label{#1}}
\makeatother
\newcommand{\newNumber}[2][]{\refstepcounter{\the\displayedcounter}%
\ifx#1\empty\label{#2}\else\gdef\envName{#1}\namedLabel{#2}\fi}
\def\newType#1#2{\newtheorem{#1}[\the\displayedcounter]{#2\gdef\envName{#2}}}
\newType{lemma}{Lemma}
\newType{proposition}{Proposition}
\newType{corollary}{Corollary}
\newType{theorem}{Theorem}
\theoremstyle{definition}
\newType{definition}{Definition}
\newType{example}{Example}
\newType{remark}{Remark}

\newcommand
{\eqncount}{\setcounter{equation}{\value{dummy}}%
\addtocounter{dummy}{1}}

\newcommand{\cD}{\mathcal D}

\newcommand{\cM}{\mathcal M}

\newcommand{\bZ}{\mathbf Z}
\newcommand{\bA}{\mathbf A}
\newcommand{\bB}{\mathbf B}

\newcommand{\DG}{\cD(G)}
\newcommand{\ab}{{\mathcal Ab}}
\newcommand\RMod{R\textrm{-Mod}}
\newcommand{\Ab}{{\mathcal Ab}}

\newcommand{\disjointunion}{\bigsqcup}
\newcommand{\smdisjointunion}{\sqcup}
\newcommand{\vv}{\, | \,}
\DeclareMathOperator{\Hom}{Hom}

\DeclareMathOperator{\invlim}{\lower 6pt\hbox{$\stackrel{\displaystyle{\lim}}{\leftarrow}$}}

\DeclareMathOperator{\Or}{\mathbf{Or}}
\DeclareMathOperator{\Ind}{Ind}
\DeclareMathOperator{\Res}{Res}
\DeclareMathOperator{\Iso}{Iso}
\DeclareMathOperator{\subInd}{\Ind_{\bB}}
\DeclareMathOperator{\subRes}{\Res_{\bB}}
\DeclareMathOperator{\subIso}{\Iso_{\bB}}

\newif\ifmathCheck\mathChecktrue

\def\cDdot{\cD_{\ast}}
\def\bAdot{\bA_\bullet}
\def\congtran{\tau}
\def\congtrandot{\tau^{\bullet}}
\def\upi{i^{\bullet}}\def\downi{i_{\bullet}}
\def\upj{j^{\bullet}}\def\downj{j_{\bullet}}

\newcommand{\OrtoDdot}{\mathfrak{o}}
\newcommand{\DtoAdot}{j}

\def\DtoDdot{\mu^\prime}
\def\DdotD{\mu}

\def\GR{{\bf gr}}

\def\Burnside{{\mathbf A}}
\def\subBurnside{{\mathbf B}}
\def\subBurnsideDot{{\mathbf B}_{\bullet}}
\newdimen\bisetdim\bisetdim=4pt
\def\biset#1#2#3{{\lower \bisetdim\hbox{$\scriptstyle#1$}}%
{{#2}}{\lower \bisetdim\hbox{$\scriptstyle#3$}}}
\def\defininghomo#1#2{{\gamma_{_{#1,#2}}}}
\def\superdefininghomo#1#2#3{{\gamma^{#3}_{_{#1,#2}}}}
\def\standardformpoint#1#2{\{#1, #2\}}
\def\standardformtoX#1{{\Psi_{#1}}}
\def\quotientpoint#1#2{{[#1,#2]}}
\def\standardRes#1{\Res_{#1}}
\def\standardInd#1{\Ind_{#1}}
\def\standardLIso#1{\mathfrak R_{#1}}
\def\standardRIso#1{\mathfrak L_{#1}}
\def\changepointbiset#1#2{V_{#1 #2}}
\def\changepointbisetK#1#2{W_{#1 #2}}
\def\doublecoset#1#2#3{#1\backslash #2/#3}
\def\mackeyComponent#1{\Phi_{#1}}
\newcommand{\id}{\mathrm{id}}

\newcommand{\leftsub}[2]{{\vphantom{#2}}_{ #1}{\hskip-1pt#2}}
\def\mackeyComponent#1{\Phi_{#1}}
\def\GsetInd#1#2{\text{Ind}_{#1}^{#2}}
\def\GsetRes#1#2{\text{Res}_{#2}^{#1}}
\def\GsetRIso#1#2{\mathfrak L_{#1}^{#2}(g_x)}

\def\cmmA#1{\cM_{\ast}\left(#1\right)}
\def\cmmB#1{\cM^{\ast}\left(#1\right)}
\def\Ax{K_{x_2}\cap L^{h_2^{-1}}_{x_1}}
\def\Bx{L_{x_1}\cap K^{h_2}_{x_2}}
\newcommand{\nr}[1]{\smallskip\noindent{\bf (#1)}.}

\begin{document}

\title[Mackey functors and bisets] 
{Mackey functors and bisets}
\author{I.~Hambleton}
\thanks{\hskip -11pt  Research partially supported by
NSERC Discovery Grant A4000 and the NSF. The authors would also like to thank the SFB 478, M\"unster, for its hospitality and support in June 2008.}
\address{Department of Mathematics \& Statistics
\newline\indent
McMaster University
\newline\indent
Hamilton, ON  L8S 4K1, Canada}
\email{ian{@}math.mcmaster.ca}
\author{L.~R.~Taylor}
\address{Department of Mathematics
\newline\indent
University of Notre Dame
\newline\indent
Notre Dame, IN 46556, USA}
\email{taylor.2@nd.edu}
\author{E.~B.~Williams}
\address{Department of Mathematics
\newline\indent
University of Notre Dame
\newline\indent
Notre Dame, IN 46556, USA}
\email{williams.4@nd.edu}

\date{Jan.~7, 2010 (revision)}
\begin{abstract}\noindent
For any finite group $G$, we define a bivariant functor from the 
Dress category of finite $G$-sets to the conjugation biset category, 
whose objects are subgroups of $G$, and whose morphisms are generated 
by certain bifree bisets. 
Any additive functor from the conjugation biset category to abelian groups 
yields a Mackey functor  by composition.
We characterize the Mackey functors which arise in this way. 
\end{abstract}

\maketitle
\section{Introduction}\label{one}
Let $G$ be a finite group.  
A Mackey functor in the sense of Dress \cite{dress2}*{p.~301} is a bivariant functor 
$$\cM \colon \DG \to \ab$$  from the category of finite left $G$-sets and 
$G$-maps to abelian groups, satisfying a pull-back axiom (M1) and 
an additivity axiom (M2). 
These axioms express the classical Mackey properties from 
representation theory, as formulated by Green \cite{green1}.  In this paper, unless otherwise mentioned, by a \emph{Mackey functor} we mean a Mackey functor in the sense of Dress.

Many of the Mackey functors encountered in applications of the theory 
factor through the $G$-Burnside category $\bA(G)$, whose objects are 
the subgroups $H\subset G$ and whose morphisms are generated by 
$H_2$\,-$H_1$ bisets with certain properties (compare \cite{htw3}*{1.A.4}).
Let $\bAdot(G)$ denote the additive completion of this category, 
defined in Section \ref{three}.

In Section \ref{six} we define a subcategory 
$\subBurnsideDot(G) \subset \bAdot(G)$, where the morphisms 
are restricted to be \emph{conjugation bisets}, and construct a 
bivariant functor $j\colon \DG \to \subBurnsideDot(G)$
(see Definition \ref{DtoBurnside}).   
A Mackey functor $\cM$ is said to be \emph{conjugation invariant} provided 
that the centralizer $C_G(H)$ acts trivially on $M(G/H)$ for all $H \subset G$, 
via the $G$-maps $\varphi\colon G/H \to G/H$ given by 
$eH \mapsto zH$, for some $z \in C_G(H)$ 
(see Definition \ref{def:conjugation invariant}). 

The main result (see Theorem \ref{Mackeyrecognition}) is a 
recognition principle for such Mackey functors. 

\begin{thma} Let $G$ be a finite group. 
A Mackey functor $\cM\colon \DG \to \ab$ factors through 
$j\colon \DG \to \subBurnsideDot(G)$ if and only if 
$\cM$ is conjugation invariant.
\end{thma}

The applications surveyed in \cite{htw3} and \cite{htw2007}, 
mostly depend on the following immediate consequence of 
our main result.

\begin{corb} For any additive functor $F\colon \bAdot(G) \to \ab$, 
the composition $F\circ j \colon \DG \to \ab$ is a Mackey functor.
\end{corb}

This result applies to classical Mackey functors  such as the 
Swan ring and the Dress ring, and  to algebraic $K$-theory 
and $L$-theory functors encountered in geometric topology 
and surgery theory. 

\begin{remark} The topic of Mackey functors has been extensively explored over 
the last 30 years, and there is a large literature (see, for example, the survey by Webb \cite{webb3}). Here are some remarks about a selection of this earlier work.

\nr{a} The idea of reformulating the classical Mackey bivariant functor  properties  
as a single functor out of an intermediate category has been carried out 
by several authors (see Lindner \cite{lindner1}*{Theorem 2}, Gaunce Lewis   \cite{gaunce_lewis1}, tom Dieck
\cite{tomDieck2}*{Chap.~IV.8}, Th\'evenaz-Webb \cite{thevenaz-webb2}). 

\nr{b} The use of bisets as morphisms in a category  has also appeared in various settings in the literature (see, for example,  Lewis-May-McClure \cite{lewis-may-mcclure1}, Adams-Gunawardena-Miller \cite{adams-gunawardena-miller1}*{\S9, p.~454}, Hambleton-Taylor-Williams
\cite{htw3}*{1.A.4}, and Bouc \cite{bouc0}, \cite{bouc2}*{\S 2}, \cite{bouc3}). 

\nr{c} In \cite{dress2}*{pp.~292, 302} Dress restricts attention to a variant of classical Mackey functors, now usually called \emph{global Mackey functors} \cite{webb2}*{p.~267}, defined on the category $\GR$  of finite groups and monomorphisms, and in particular the maps induced by conjugations
depend only on the underlying group homomorphisms. 
 In \cite{dress2}*{p.~298-299}, Dress  describes the passage from global Mackey functors to his functors. The output of this process is a conjugation-invariant Mackey functor.

\nr{d} 
In \cite{webb2}*{p.~271},  Webb  sketches a proof of a result analogous to Theorem A,  that a global Mackey functor is equivalent to an additive functor out of a category $\Omega_{\bZ}$ whose objects are finite groups, and whose morphisms are bifree bisets. The restriction of $\Omega_{\bZ}$ to the subgroups  of a fixed group $G$ is just the category $\bAdot(G)$, defined in our MSRI preprint (1990) \cite{htw4}*{Ex.~5.5}.

\nr{e} There is a sub-category of $\GR$ defined by restricting to the subgroups of a fixed group $G$, and to the monomorphisms generated by inclusions and conjugations. A  Mackey functor out of this sub-category  is equivalent to an additive functor out of $\subBurnsideDot(G)$. This is exactly the statement of Theorem A expressed in classical terms. In other words, Theorem A provides a direct method to check the Mackey properties, starting from  minimal input data. 

\nr{f} The proof of Theorem A follows Webb's proof \cite{webb2}*{p.~271} in outline, but we supply all the technical details necessary, for example, to check that we have actually constructed a  functor $F\colon \subBurnsideDot(G) \to \Ab$, factoring a given conjugation invariant Mackey functor.
\end{remark}

\section{Mackey functors}\label{two}
We will first recall the basic definitions Dress used in his formulation of
induction theory \cite{dress2}.
Let $G$ be a finite group, and let 
$\DG$ denote the category whose objects are finite, 
left $G$-sets and whose morphisms are $G$-maps. 
A \emph{Mackey functor} is a bivariant functor
$\cM = (\cM_\ast, \cM^\ast) \colon \cD(G) \to \ab$, 
where $\ab$ denotes the category of abelian groups and groups 
homomorphisms, such that $\cM_\ast(S) = \cM^\ast(S)$ 
for each object $S \in \DG$, and the following two  properties hold:

\medskip\noindent
\textbf{(M1)} For any pullback diagram of finite $G$-sets
$${\vcenter
{\xymatrix
{S\ar[r]^\Psi\ar[d]_\Phi&S_1\ar[d]^\varphi\\S_2\ar[r]^\psi&T}}}
$$
the induced maps give an commutative diagram
$${\vcenter
{\xymatrix
{\cM(S)\ar[r]^{\Psi_\cM}&\cM(S_1)\\
\cM(S_2)\ar[r]^{\psi_\cM}\ar[u]^{\Phi^\cM}&\cM(T)\ar[u]_{\varphi^\cM}}}}
$$
Here we denote the covariant maps by $\psi_\cM$ and the
contravariant maps by $\varphi^\cM$.

\medskip\noindent
\textbf{(M2)} The  embeddings of $S_1$ and $S_2$ into the disjoint
union $S_1\disjointunion S_2$ define an isomorphism
$\cM^\ast(S_1\disjointunion S_2) \to \cM^\ast(S_1)\oplus \cM^\ast(S_2)$. 
Let $\cM (\emptyset) = 0$.

\medskip
The property (M1) is the usual double coset formula, and (M2) gives additivity.
We remark that for any bivariant functor satisfying (M1), the composition
$\cM_\ast(S_1)\oplus \cM_\ast(S_2) \to \cM_\ast(S_1\disjointunion S_2) = 
\cM^\ast(S_1\disjointunion S_2) \to \cM^\ast(S_1)\oplus \cM^\ast(M_2) = 
\cM_\ast(S_1)\oplus \cM_\ast(S_2)$ 
is just the identity matrix.
It follows that  any sub-bivariant functor of a Mackey functor is Mackey. 

\begin{remark} We could replace the target category $\ab$ of abelian groups 
throughout by the category $\RMod$ of $R$-modules, for any commutative ring 
$R$ with unit.
\end{remark}

\section{The $G$-Burnside category $\Burnside(G)$}\label{three} 
We will be interested in  Mackey functors which factor through the 
$G$-Burnside category $\Burnside(G)$, whose
objects are subgroups $H\subset G$, and where
$\Hom_{\Burnside(G)}(H_1, H_2)$ is the Grothendieck construction 
applied to the isomorphism classes of finite bifree $H_2$\,-$H_1$ bisets
(meaning both left and right actions are free). In contrast, the morphisms in  the $G$-Burnside category of \cite{adams-gunawardena-miller1}*{\S9, p.~454}, and our category $RG$-Morita \cite{htw3}*{1.A.4}  just have a one-sided isotropy assumption.

To make our recognition principle more precise, we will define a subcategory 
$\subBurnside(G)\subset \Burnside(G)$ by restricting its  morphisms to 
\emph{conjugation bisets} (see Definition \ref{conjugationbisets}).

Because of the Grothendieck construction, $\Burnside(G)$ and 
$\subBurnside(G)$ are both Ab--categories: the morphism sets are
abelian groups and the compositions are bilinear  \cite{maclane2}*{I.8, p.~28}.
Let $u\colon \bA(G) \to \bAdot(G)$ and 
$u\colon \subBurnside(G) \to \subBurnsideDot(G)$ 
denote the associated universal free additive categories, and  
the universal inclusions (see  \cite{maclane2}*{VII.2, problem 6, p.~194}).

It turns out that the Mackey functors which factor through 
$\subBurnsideDot(G)$ have an additional property, 
called \emph{conjugation invariance}, which can be expressed 
in terms of their restriction to the orbit category $\Or(G)$.
Recall that the objects of $\Or(G)$ are the subgroups $H\subset G$, 
and the morphisms $\Hom_{\Or(G)}(H_1, H_2)$ are the $G$-maps 
$\varphi\colon G/H_1 \to G/H_2$. 
Any such $G$-map is uniquely determined by $eH_1 \mapsto gH_2$, 
where $g^{-1}H_1g \subseteq H_2$. 
If $H_1 = H_2 = H$, then any element $z \in C_G(H)$ in the centralizer 
of $H$ gives a $G$-map $\varphi_z\colon G/H \to G/H$.

\begin{definition}\label{def:conjugation invariant}
A functor $F\colon \Or(G) \to \Ab$ is called \emph{conjugation invariant} 
if  the induced maps $F(\varphi_z) = \id\colon F(H) \to F(H)$, whenever
$\varphi_z\colon G/H \to G/H$ is given by $eH \mapsto zH$, for some 
$z \in C_G(H)$. 
A Mackey functor is said to be conjugation invariant if its restriction to 
$\Or(G)$ satisfies this condition.
\end{definition}

We now relate this condition to the conjugation homomorphisms 
which will be used in the definition of $\subBurnsideDot(G)$. 
If $\varphi\in \Hom_{\Or(G)}(H_1, H_2)$ is represented by $g\in G$, 
we let $c_g\colon H_1 \to H_2$ denote the associated 
conjugation homomorphism given by $c_g(h_1) = g^{-1}h_1g \in H_2$, 
for all $h_1 \in H_1$. 
Since the $G$-map $\varphi$ only depends on the coset $gH_2$, 
we may vary $g \sim gh_2$, for any $h_2 \in H_2$. 
The associated conjugation homomorphism 
$c_{gh_2} = c_{h_2}\circ c_g$ is thus well-defined 
(as a homomorphism) up to conjugation by elements of $H_2$. 
Let
$$c_\varphi:= [c_g] \in \Hom(H_1, H_2)/(\text{conjugation\ in\ } H_2)$$
denote the equivalence class of the associated conjugation homomorphism 
to a $G$-map $\varphi(eH_1) = gH_2$.
Two different morphisms $\varphi_1, \varphi_2 \in \Hom_{\Or(G)}(H_1, H_2)$ 
yield the same equivalence class $c_{\varphi_1} = c_{\varphi_2}$ if and only if
$g_1h_2g_2^{-1}$ centralizes $H_1$, for some $h_2 \in H_2$. 

\begin{lemma}\label{conginv}
A functor $F\colon \Or(G) \to \Ab$ is conjugation invariant if and only if
$c_{\varphi_1} = c_{\varphi_2}$ implies $F(\varphi_1) = F(\varphi_2)$
for all morphisms 
$\varphi_1, \varphi_2  \in \Hom_{\Or(G)}(H_1, H_2)$.
\end{lemma}

\begin{proof}
Suppose that $\varphi_1, \varphi_2 \in \Hom_{\Or(G)}(H_1, H_2)$ are given by 
$\varphi_i(eH_1) = g_iH_2$, for $i = 1,2$. 
If there exists an element $h_2 \in H_2$ such that 
$z:=g_1h_2g_2^{-1}\in C_G(H_1)$,
then we see that $\varphi_1 = \varphi_2 \circ \varphi_z$. 
Therefore, if $F$ is conjugation invariant and $c_{\varphi_1} = c_{\varphi_2}$ 
we conclude that $F(\varphi_1) = F(\varphi_2)$. 
Conversely, for any subgroup $H \subset G$, and any  $z\in C_G(H)$, 
the $G$-maps $\varphi_z, \id\colon G/H \to G/H$ have the property
$c_{\varphi_z} = c_{\id}$. 
Therefore the given condition implies that $F(\varphi_z) = F(\id) = \id$, 
and hence $F$ is conjugation invariant.
\end{proof}

The morphisms in $\Burnside(G)$ are defined by  the Grothendieck 
construction with addition operation the disjoint union of bisets. 
By convention, the empty biset  $\emptyset$ represents the zero element.
Composition comes from the balanced product: 
$${_{H_3}}X_{H_2} \circ {_{H_2}}Y_{H_1} = 
({_{H_3}}X_{H_2})\times_{H_2} ({_{H_2}}Y_{H_1}).$$
The reader should check that this is well--defined 
on isomorphism classes of bisets and ``bilinear'' in that 
$$({_{H_3}}X_{H_2} \disjointunion {_{H_3}}Y_{H_2}) 
\circ {_{H_2}}Z_{H_1}\cong ({_{H_3}}X_{H_2} 
\circ{_{H_2}}Z_{H_1})\disjointunion\,
({_{H_3}}Y_{H_2}\circ {_{H_2}}Z_{H_1})$$ 
with a similar formula for disjoint union on the right. 
The morphisms in $\bAdot(G)$ are matrices of morphisms in $\Burnside(G)$.

\begin{definition}
We define a contravariant involution 
$\congtran\colon \Burnside(G) \to \Burnside(G)$,  
by the identity on objects, and on morphisms it is the map induced 
on the Grothendieck construction by the function
which takes the finite bifree $H_2$\,-$H_1$ biset ${_{H_2}}X_{H_1}$ 
to the  finite bifree $H_1$\,-$H_2$ biset ${_{H_1}}X_{H_2}$ which is 
$X$ as a set and $h_1\cdot x \cdot h_2$ is defined to be $h_2^{-1}x h_1^{-1}$.
\end{definition}

The reader needs to check that isomorphic bisets are isomorphic after reversing the order,
and should also check that the transpose conjugate of a disjoint union is isomorphic to the
disjoint union of the conjugate transposes of the pieces.
This means that $\congtran$ is a functor which induces a homomorphism of Hom--sets.
It is clearly an involution, not just up to natural equivalence.
Since $\congtran$ is a homomorphism on Hom-sets, it induces an additive contravariant
involution $\congtrandot\colon \bAdot(G) \to \bAdot(G)$, called \emph{conjugate transpose}, 
which commutes with the functor $u\colon \bA(G) \to \bAdot(G)$.  
By definition,  $\congtrandot$ acts on a matrix of morphisms by applying 
$\congtran$ to each entry, and then transposing the matrix.

\section{Indecomposable $H_2$\,-$H_1$ bisets}\label{four}
We will need some more detailed information about the bifree bisets used 
to define morphisms in $\Burnside(G)$. Much of this material can be found in Bouc \cite{bouc0}, but we include the details here to fix our notation and to emphasize the role of the base-points.

An $H_2$\,-$H_1$ biset is just a left $(H_2\times H^{\rm op}_1)$-set and so
any finite $H_2$\,-$H_1$-biset is a disjoint union of transitive 
$(H_2\times H^{\rm op}_1)$-sets. 
Since there are three groups acting 
($H_1$, $H_2$ and  $H_2\times H^{\rm op}_1$), the following definition 
should avoid confusion.
\begin{definition} 
An $H_2$\,-$H_1$ biset  $X$ is \emph{indecomposable} if $X$ is a
transitive $(H_2\times H^{\rm op}_1)$-set.
\end{definition}

Since every $H_2$\,-$H_1$ biset is a disjoint union of indecomposable 
$H_2$\,-$H_1$ bisets, it  follows that $X_1 \disjointunion Y$ is isomorphic 
to $X_1 \disjointunion Y$ if and only if
$X_1$ is isomorphic to $X_2$.
One result of this remark is that the Grothendieck group of finite bifree 
$H_2$\,-$H_1$ bisets is the free abelian group on the indecomposable 
bifree ones.

An indecomposable $H_2$\,-$H_1$ biset $X$ is isomorphic 
(via choice of base point) to a coset space
$(H_2\times H^{\rm op}_1)/S$, for some subgroup  
$S \subset H_2\times H^{\rm op}_1$. 
These models will be used extensively below, so we make some remarks 
and introduce some notation.

\begin{enumerate}\addtolength{\itemsep}{0.3\baselineskip}
\item
To see $(H_2\times H^{\rm op}_1)/S$ as an $H_2$\,-$H_1$ biset, define the
$H_2$ action to be left multiplication in the first coordinate, but define
$(h_2,h_{1})S \cdot  g_{1}= (h_2, g^{-1}_{1} h_{1}) S$. 

\item We introduce the following notation for points in $(H_2\times H^{\rm op}_1)/S$:
if $h_2\in H_2$ and $h_1\in H_1$, write
$\standardformpoint{h_2}{h_1} = (h_2,h^{-1}_1)S$.
In this notation, the $H_2$\,-$H_1$ action is the evident one:
$g_2 \standardformpoint{h_2}{h_1} g_1 = \standardformpoint{g_2 h_2}{h_1 g_1}$.

\item The left isotropy group of $(e,e)$ is just $(H_2\times\{e\}) \cap S
$, and the
right isotropy group of $(e,e)$ is just $(\{e\}\times H^{\rm op}_1) \cap S$, so
the actions are free if and only if
$(H_2\times\{e\}) \cap S = \{e\times e\}$ and
$(\{e\} \times H^{\rm op}_1)\cap S = \{e\times e\}$.

\item Equivalently, the $S$-action is bifree if the compositions
$\iota_1\colon S \subset H_2\times H^{\rm op}_1 \to H_2$ and
$S \subset H_2\times H^{\rm op}_1 \to H^{\rm op}_1$
are injective. 
We will work instead with the composite
$\iota_2\colon S \to H_2\times H^{\rm op}_1 \to H^{\rm op}_1 \xrightarrow{\approx} H_1$,
where the last homomorphism takes $h_1$ to $h_1^{-1}$.

\item Let $L \subset H_2$ denote the image of the isotropy subgroup $S$ 
under the injection $\iota_1$, and let $K \subset H_1$ denote the image of $S$ 
under the injection $\iota_2$. 
Define
$$\gamma\colon L  \xrightarrow{\ \iota_1^{-1}\ } S \xrightarrow{\ \iota_2\ } K$$
and notice that this is an isomorphism.

\item Conversely, given subgroups $L \subset H_2$, $K\subset H_1$, 
and an isomorphism $\gamma\colon L \xrightarrow{\approx} K$, 
let $\biset{H_2}{[L, \gamma, K]}{{H_1}}$
denote the indecomposable bifree biset $(H_2\times H^{\rm op}_1)/S$, 
where $S \subset H_2\times H^{\rm op}_1$
is the graph of
$$\bar{\gamma} \colon  L \xrightarrow{\gamma} K \subset H_1 \xrightarrow{\iota} H_1^{\rm op},$$
and $\iota\colon H_1 \xrightarrow{\approx} H^{\rm op}_1$ is 
the usual isomorphism which takes $h$ to $h^{-1}$.
\end{enumerate}

\medskip
Given an indecomposable $H_2$\,-$H_1$ biset $X$, 
a choice of base-point $x\in X$ yields an isotropy
subgroup $S_x\subset H_2\times H^{\rm op}_1$ 
and a preferred biset isomorphism
\eqncount
\begin{equation}
\standardformtoX{x}\colon (H_2\times H^{\rm op}_1)/S_x \xrightarrow{\cong} X
\end{equation}
defined by 
$\standardformtoX{x}\bigl(\standardformpoint {h_2}{ h_1}\bigr) = h_2 x h_1$.

\begin{definition}[The standard representation]\namedLabel{the standard representation}
Let $X$ be an indecomposable bifree $H_2$\,-$H_1$ biset, and $x \in X$ a base-point. 
The \emph{standard representation for $X$ at $x$} is the data
\newNumber[Formula]{standard form}
$$\standardformtoX{x}\colon \biset{H_2}{[L_x, \defininghomo{X}{x}, K_x]}{{H_1}} \xrightarrow{\cong} X
\leqno(\ref{standard form})$$
given by the preferred biset isomorphism
$\standardformtoX{x}\bigl(\standardformpoint {h_2}{ h_1}\bigr) = h_2 x h_1$, where 
$L_x\subset H_2$ is  the image
of the isotropy subgroup $S_x$ under $\iota_1$,  $K_x\subset H_1$ 
is the image of $\iota_2$, and 
$$\defininghomo{X}{x}\colon L_x  \xrightarrow{\ \iota_1^{-1}\ } S_x \xrightarrow{\ \iota_2\ } K_x$$
is an isomorphism. 
We noted above that $L_x$, $K_x$ and $\defininghomo{X}{x}$ determine $S_x$
as a (graph) subgroup of $ H_2\times H^{\rm op}_1$. \qed
\end{definition}
\begin{remark}
Any model coset space
$\biset{H_2}{[L, \gamma, K]}{{H_1}}$ is the standard representation for 
some indecomposable bifree biset. 
We let $X = (H_2\times H^{\rm op}_1)/S$, and choose $x\in X$ 
as the coset of the identity,
then $L_x = L$, $K_x = K$ and $\defininghomo{X}{x} = \gamma$. \qed
\end{remark}

\begin{remark}\label{define_parts}
Here is a second description of the standard representation 
(compare \cite{bouc-thevenaz2}*{5.1}).
We present it so as to identify the component of any point $x$ 
in an $H_2$\,-$H_1$ bifree biset $X$, indecomposable or not.
Let $L_x = \{ h_2 \in H_2\ \vert\ h_2 x \in x H_1\ \}$.
Check $L_x$ is a subgroup of $H_2$.
Since the $H_1$ action is free, there exists a unique $h_1\in H_1$ 
such that $h_2 x = x h_1$.
Define a function $f \colon L_x \to H_1$ by $f(h_2) = h_1$.
Let $K_x$ denote the image of $f$ and let
$\defininghomo{X}{x} \colon L_x \to K_x$ denote the restriction of $f$. 
Check that $f$ is a homomorphism and hence $\defininghomo{X}{x}$ 
is an isomorphism.
Finally check that
$\standardformtoX{x}\colon \biset{H_2}{[L_x, \defininghomo{X}{x}, K_x]}{{H_1}} \to X$
defined by \namedRef{standard form} above
is an injection of $H_2$\,-$H_1$ bisets which is
then automatically onto the component of $X$ containing $x$. \qed
\end{remark}
Hereafter we will call $\biset{H_2}{[L_x, \defininghomo{X}{x},K_x]}{H_1}$ the
\emph{standard representation of $X$ at $x$}
even if $X$ is not indecomposable.
It follows that
$$\mathop{\disjointunion}_{x\in H_2\backslash X/H_1} \  %
\biset{H_2}{[L_x, \defininghomo{X}{x}, K_x]}{{H_1}}
\xrightarrow{\quad\disjointunion \standardformtoX{x} \quad} X$$
is a bijection of $H_2$\,-$H_1$ bisets, where we take one $x$ 
in each indecomposable component of $X$. 
Here are two useful properties of the standard representation for the reader to verify.

\medskip
We note the effect of the involution $\congtran\colon\bA(G) \to \bA(G)$ 
on the standard representations.

\begin{lemma}\namedLabel{standard representation for conjugate}
If $[L_x,\defininghomo{X}{x}, K_x]$ is
the standard representation for $X$ at $x$, then
$[K_x, \superdefininghomo{X}{x}{-1},L_x]$  is
the  standard representation for $\congtran(X)$ at $x$.
\end{lemma}

Next we consider a change of base-point.
If $y$ is another point in the same indecomposable component as $x$,
then $y = h_2 x h_1$ for some choice of
$h_2\in H_2$ and $h_1\in H_1$.
For $g\in G$ and any subgroup $K\subset G$, let $K^g = g^{-1} K g$.
If $\gamma\colon L\to K$
is a homomorphism between subgroups of $G$ and if $g_1$, $g_2\in G$, define
$\leftsub{g_2}{\gamma}^{g_1}(h) = g^{-1}_1 \gamma( g_2^{-1} h g_2) g_1
\colon L^{g_2^{-1}} \to K^{g_1}$.
\begin{lemma}\namedLabel{base point change}
With notation as above,
the standard representation for $X$ at $y = h_2 x h_1 \in X$ is
$\bigl[L^{h_2^{-1}}_x,  \leftsub{h_2}{\superdefininghomo{X}{x}{h_1}}  ,K_x^{h_1}\bigr]$.
In other words, $L_y = L^{h_2^{-1}}_x$, $K_y = K_x^{h_1}$ and
$\defininghomo{X}{y} =  \leftsub{h_2}{\superdefininghomo{X}{x}{h_1}}$. 
\end{lemma}

\medskip
It will be useful later to be able to identify $X/H_1$ and $H_2\backslash X$.

\begin{lemma} \namedLabel{action orbits}
If $X$ is an indecomposable, bifree $H_2$\,-$H_1$ biset, then
the bijection of bisets 
$\standardformtoX{x}\colon \biset{H_2}{[L_x, \defininghomo{X}{x}, K_x]}{{H_1}} \to X$
induces
\begin{enumerate}
\item a bijection of left $H_2$-sets $H_2/ L_x \to X/H_1$
\item a bijection of right $H_1$-sets
$K_x\backslash H_1 \to H_2\backslash X$.
\end{enumerate}
\end{lemma}
\begin{proof}
The proof is immediate.
\end{proof}

\section{Composition of bisets}\label{five}
Fix indecomposable, bifree bisets $\biset{H_2}{(X_1)}{H_1}$ and
$\biset{H_3}{(X_2)}{H_2}$ and let us analyze
$X_3 = X_2\times_{H_2} X_1$.
In general $X_3$ will not be indecomposable, and the Mackey double coset
formula enters the picture.
We begin by analyzing the standard representation at a point.
\begin{lemma}\namedLabel{balanced product components}
The standard representation for $X_3=X_2\times_{_{H_2}} X_1$ at 
$\quotientpoint{x_2}{x_1}$ is given by
\begin{itemize}\namedLabel{standard representation of a product}
\item $L_{\quotientpoint{x_2}{x_1}} = 
\superdefininghomo{X_2}{x_2}{-1}\bigl(K_{x_2} \cap L_{x_1}\bigr)$,
\item $K_{\quotientpoint{x_2}{x_1}}=
\defininghomo{X_3}{\quotientpoint{x_2}{x_1}}(L_{\quotientpoint{x_2}{x_1}}) = 
\defininghomo{X_1}{x_1}\bigl(K_{x_2} \cap L_{x_1}\bigr)$ 
\item $\defininghomo{X_3}{\quotientpoint{x_2}{x_1}} =
\defininghomo{X_1}{x_1} \circ \defininghomo{X_2}{x_2}$.
\end{itemize}
\end{lemma}
\begin{proof} 
Since the projection $X_2\times X_1 \to X_2\times_{H_2} X_1$ is onto,
every point in $X_2\times_{H_2} X_1$ is the image of a point in $X_2 \times X_1$.
Given $x_1\in X_1$ and $x_2\in X_2$ write $\quotientpoint{x_2}{x_1}
\in X_3 = X_2\times_{H_2} X_1$ for the image of $x_2\times x_1$.
In this notation, the $H_3$\,-$H_1$ structure is the evident one:
$h_3 \quotientpoint{x_2}{x_1} h_1 = \quotientpoint{h_3 x_2}{x_1 h_1}$.

The point $\quotientpoint{x_2}{x_1}\in X_3$ is a set of points 
$\{x_2 h \times h^{-1} x_1\} \subset X_2 \times X_1$
for all $h \in H_2$.
If $h_3 (x_2 \times x_1) h_1 = x_2 h \times h^{-1} x_1 \in X_2 \times X_1$ 
for some $h \in H_2$ then $h_3 x_2 = x_2 h$ and $h^{-1} x_1 = x_1 h_1$.
Now if $h_3 x_2 = x_2 h$ then $h = \defininghomo{X_2}{x_2}(h_3)$
and $h_3 \in L_{x_2}$.
If $h^{-1} x_1 = x_1 h_1$ then $h_1 = \defininghomo{X_1}{x_1}(h^{-1})$
and $h^{-1} \in L_{x_1}$.
Since $L_{x_1}$ is a group, $h \in L_{x_1}$.
Since $h = \defininghomo{X_2}{x_2}(h_3)$, $h \in K_{x_2}$.
Therefore,
$h_3 \in \superdefininghomo{X_2}{x_2}{-1}\bigl(K_{x_2} \cap L_{x_1}\bigr)$.
Let $L_{\quotientpoint{x_2}{x_1}} = 
\superdefininghomo{X_2}{x_2}{-1}\bigl(K_{x_2} \cap L_{x_1}\bigr)$
and define
$\defininghomo{X_3}{\quotientpoint{x_2}{x_1}} =
\defininghomo{X_1}{x_1} \circ \defininghomo{X_2}{x_2}$.
What we have seen so far is that if
$h_3\ \quotientpoint{x_2}{x_1}\ h_1 = 
\quotientpoint{x_2}{x_1}$ in $X_3$, then $h_3 \in L_{\quotientpoint{x_2}{x_1}}$ and
$h_1 = \defininghomo{X_3}{\quotientpoint{x_2}{x_1}}(h_3)$.

If $h_3 \in L_{\quotientpoint{x_2}{x_1}}$, then
$h_3 (x_2 \times x_1)\defininghomo{X_3}{\quotientpoint{x_2}{x_1}}(h_3) =
x_2 \defininghomo{X_2}{x_2}(h_3) \times \defininghomo{X_2}{x_2}(h_3)^{-1} x_1$
and it further follows that for any $h_2 \in H_2$,
$h_3 (x_2 h_2 \times h^{-1}_2 x_1)\defininghomo{X_3}{\quotientpoint{x_2}{x_1}}(h_3) =
x_2 \defininghomo{X_2}{x_2}(h_3)h_2 \times h^{-1}_2\defininghomo{X_2}{x_2}(h_3)^{-1} x_1 =
x_2 \bigl(\defininghomo{X_2}{x_2}(h_3)h_2\bigr) \times
\bigl(\defininghomo{X_2}{x_2}(h_3) h_2\bigr)^{-1} x_1$.
In other words, if $L_{\quotientpoint{x_2}{x_1}}$ and 
$\defininghomo{X_3}{\quotientpoint{x_2}{x_1}}$ are defined as above,
then for any $h_3 \in L_{\quotientpoint{x_2}{x_1}}$, 
$h_3\ \quotientpoint{x_2}{x_1}\ \defininghomo{X_3}{\quotientpoint{x_2}{x_1}}(h_3) = 
\quotientpoint{x_2}{x_1}$.
If we let $K_{\quotientpoint{x_2}{x_1}} =
\defininghomo{X_3}{\quotientpoint{x_2}{x_1}}(L_{\quotientpoint{x_2}{x_1}})$,
then the proof is complete.
\end{proof}
\begin{lemma}\namedLabel{parameter}
The double coset space $K_{x_2}\backslash H_2/L_{x_1}$ parametrizes 
the indecomposable components of $X_2\times_{_{H_2}} X_1$, where $x_1\in X_1$ and $x_2\in X_2$ are base points.
\end{lemma}
\begin{proof}
To simplify the notation  let $X_2 = 
\biset{H_2}{[L_3, \gamma_2, K_2]}{H_1}$, and $X_1 = 
\biset{H_2}{[L_2, \gamma_1, K_1]}{H_1}$.
The set of indecomposable components is given by
$$H_3\backslash (X_2\times_{H_2} X_1)/H_1 =
(H_3\backslash X_2) \times_{H_2}(X_1/H_1).$$
By \namedRef{action orbits},
$(H_3\backslash X_2) \times_{H_2}(X_1/H_1) = 
(K_2\backslash H_2) \times_{H_2} (H_2/L_2) =
K_2\backslash H_2/L_2$.
If $\{ h_2\}\in K_2\backslash H_2/L_2$ is a set of double coset representatives with $h_2 \in H_2$, 
then $\quotientpoint{x_2 h_2}{x_1}$ gives a set of points, one in each 
indecomposable component.
\end{proof}

Next we give an explicit formula for $X_3=X_2\times_{_{H_2}} X_1$ as a disjoint union of indecomposable bisets (compare \cite{bouc0}*{3.2}). We define a map
 $$\mackeyComponent{h_2}
 \colon \biset{H_3}{[L_{x_3},\defininghomo{X_3}{x_3},K_{x_3}]}{H_1}
 \to
\biset{H_3}{[L_{x_2}, \defininghomo{X_2}{x_2}, K_{x_2}]}{H_2} \times_{_{\scriptstyle H_2}}
\biset{H_2}{[L_{x_1}, \defininghomo{X_1}{x_1}, K_{x_1}]}{H_1}$$
by the formula
$\mackeyComponent{h_2}(\standardformpoint{h_3}{h_1})
 =
\quotientpoint{\standardformpoint{h_3}{h_2}}{\standardformpoint{e}{h_1}}$, using the notation of \namedRef{the standard representation}.

\begin{theorem}[Mackey double coset formula]\namedLabel{The Mackey double coset formula}
The  $H_3$-$H_1$ biset $X_3=X_2\times_{_{H_2}} X_1$  is  given by the disjoint union  of the left-hand vertical maps
in the diagram

$$\xymatrix{\left [\superdefininghomo{X_2}{x_2}{-1}
\bigl(K_{x_2}\cap L_{x_1}^{h_2^{-1}}\bigr), 
\defininghomo{X_1}{x_1} \hskip-6pt\circ c_{h_2}\circ \defininghomo{X_2}{x_2}, 
\defininghomo{X_1}{x_1}\bigl(K^{h_2}_{x_2} \cap L_{x_1}\bigr)\right]
\ar@{=}[r]\ar[d]^{\mackeyComponent{h_2}}&
\biset{H_3}{[L_{x_3},
\defininghomo{X_3}{x_3},
K_{x_3}]}{H_1}
\ar[d]^{\standardformtoX{x_3}}
\\
\biset{H_3}{[L_{x_2}, \defininghomo{X_2}{x_2}, K_{x_2}]}{H_2} \times_{_{\scriptstyle H_2}}
\biset{H_2}{[L_{x_1}, \defininghomo{X_1}{x_1}, K_{x_1}]}{H_1}
\ar[r]^(0.7){\standardformtoX{x_2}\times\standardformtoX{x_1}}_(0.7){\approx}&
X_2\times_{_{H_2}} X_1
}
$$
 over the base-points $x_3 = \quotientpoint{x_2 h_2}{x_1}$, for $h_2 \in K_{x_2}\backslash H_2/L_{x_1}$. 
 \end{theorem}
\begin{proof}
We first derive the explicit formula displayed in the top left corner of the diagram.
Let $[L_{x_2},\defininghomo{X_2}{x_2},K_{x_2}]$ be the standard representation for
$X_2$ at $x_2$, and let $[L_{x_1},\defininghomo{X_1}{x_1},K_{x_1}]$ be
the standard representation for $X_1$ at $x_1$.
Then the standard representation for $X_2$ at $x_2 h_2$ is given by 
\namedRef{base point change}:
$[L_{x_2},\defininghomo{X_2}{x_2 h_2},K_{x_2}] = 
[L_{x_2},\superdefininghomo{X_2}{x_2}{h_2},K_{x_2}^{h_2}]$.
It further follows that the standard representation for $X_3=X_2\times_{_{H_2}} X_1$ 
at $x_3=\quotientpoint{x_2 h_2}{x_1}$ is given by
$$[L_{x_3}, \defininghomo{X_3}{x_3}, K_{x_3}]=[L_{[x_2 h_2,x_1]}, \defininghomo{X_3}{[x_2 h_2, x_1]}, K_{[x_2 h_2,x_1]}]~.$$
By \namedType{parameter}s \ref{base point change}, \ref{balanced product components} and \ref{parameter}
\begin{itemize}
\item $L_{[x_2 h_2,x_1]} = 
\superdefininghomo{X_2}{x_2 h_2}{-1}\bigl(K_{x_2 h_2} \cap L_{x_1}\bigr) 
= 
(\superdefininghomo{X_2}{x_2}{h_2})^{-1}\bigl(
K^{h_2}_{x_2} \cap L_{x_1}\bigr)=\superdefininghomo{X_2}{x_2}{-1}
\bigl(K_{x_2}\cap L_{x_1}^{h_2^{-1}}\bigr)$\vskip 6pt
\item $K_{[x_2 h_2,x_1]} = 
\defininghomo{X_1}{x_1}\bigl(K_{x_2 h_2} \cap L_{x_1}\bigr) = 
\defininghomo{X_1}{x_1}\bigl(K^{h_2}_{x_2} \cap L_{x_1}\bigr)$
\vskip 6pt
\item $\defininghomo{X_3}{[x_2 h_2, x_1]} = 
\defininghomo{X_1}{x_1} \hskip-6pt\circ \defininghomo{X_2}{x_2 h_2}  = 
\defininghomo{X_1}{x_1}\hskip -6pt \circ \superdefininghomo{X_2}{x_2}{h_2} =
\defininghomo{X_1}{x_1} \hskip -6pt \circ c_{h_2}\circ \defininghomo{X_2}{x_2}
$
\vskip 6pt
\end{itemize}

For $h_2\in H_2$ the function
defined by $\mackeyComponent{h_2}(\standardformpoint{h_3}{h_1}) =
\quotientpoint{\standardformpoint{h_3}{h_2}}{\standardformpoint{e}{h_1}}$
makes the following diagram commute
$$\xymatrix{
\biset{H_3}{[L_{\quotientpoint{x_2 h_2}{x_1}},
\defininghomo{X_3}{\quotientpoint{x_2 h_2}{x_1}},
K_{\quotientpoint{x_2 h_2}{x_1}}]}{H_1}
\ar[rd]_{\standardformtoX{\quotientpoint{x_2 h_2}{x_1}}}
\ar[r]^-{\mackeyComponent{h_2}}
&\biset{H_3}{[L_{x_2}, \defininghomo{X_2}{x_2}, K_{x_2}]}{H_2}
\times_{_{\scriptstyle H_2}}
\biset{H_2}{[L_{x_1}, \defininghomo{X_1}{x_1}, K_{x_1}]}{H_1}
\ar[d]^{\standardformtoX{x_2}\times\standardformtoX{x_1}}\\
&{_{_{H_3}}}\bigl(X_2\times_{_{H_2}} X_1\bigr){_{_{H_1}}}\\
}$$

Since the $\Psi$ maps are injective, it suffices to prove that the formula given for
$\mackeyComponent{h_2}$ makes the diagram commute.
Over and then down takes $\standardformpoint{h_3}{h_1}$ to
$\quotientpoint{\standardformpoint{h_3}{h_2}}{\standardformpoint{e}{h_1}}$ and then to
$\quotientpoint{h_3 x_2 h_2}{x_1 h_1}$ whereas
$\standardformtoX{\quotientpoint{x_2 h_2}{x_1}}(\standardformpoint{h_3}{h_1}) =
h_3 \quotientpoint{x_2 h_2}{x_1} h_1 = \quotientpoint{h_3 x_2 h_2}{x_1 h_1}$.

The conclusion now follows immediately from the commutativity of the displayed diagram, since the disjoint union of the right-hand vertical maps  $\standardformtoX{x_3}$ is a bijection, and the map
$\standardformtoX{x_2}\times\standardformtoX{x_1}$ is a bijection (see \namedRef{the standard representation}).
\end{proof}
\begin{remark}
The map $\defininghomo{X_3}{\quotientpoint{x_2 h_2}{x_1}}$ can be displayed as
the composition on the top row of the diagram
$$\xymatrix{
L_{\quotientpoint{x_2 h_2}{x_1}}\ar[d]
\ar[rr]^{\defininghomo{X_2}{x_2}\vrule width .5pt depth 6pt height 8pt }&
&K_{x_2}\cap L^{h_2^{-1}}_{x_1}\ar[r]^-{c_{h_2} }\ar[d]
&K^{h_2}_{x_2}\cap L_{x_1}\ar[d] \ar[rrr]^{\quad \defininghomo{X_1}{x_1}
\vrule width .5pt depth 6pt height 8pt  \quad}\ar[dr]
&&& K_{\quotientpoint{x_2 h_2}{x_1}}\ar[d]\\
L_{x_2}\ar[rr]^{\defininghomo{X_2}{x_2}} &&
K_{x_2}\ar[r]^{c_{h_2}}& K^{h_2}_{x_2}
&L_{x_1}\ar[rr]^{\quad\defininghomo{X_1}{x_1}\quad}&&K_{x_1}\\
}
$$
where $c_g(h) = g^{-1} h g$ is conjugation. 
\end{remark}

\section{Conjugation bisets and $\subBurnside(G)$}\label{six}
We will now define a subcategory $\subBurnside(G) \subset \Burnside(G)$, 
with the same objects (the subgroups of $G$), 
but with morphisms  restricted to bifree conjugation bisets. 
As before, we perform the Grothendieck construction on the isomorphism 
classes of these bisets to get an $\Ab$-category. 
The universal construction
$u\colon \subBurnside(G) \to \subBurnsideDot(G)$ 
then produces an additive subcategory (with involution) of $\bAdot(G)$.
\begin{definition}[Conjugation bisets]\label{conjugationbisets}
An indecomposable \emph{conjugation biset} is
an indecomposable bifree biset $\biset{H_2}{X}{H_1}$
so that the isomorphism in the standard representation of $X$ at $x\in X$,
$\defininghomo{X}{x} = c_g$, for some $g\in G$ such that $g^{-1}L_xg= K_x$.
A \emph{conjugation biset} is a bifree biset, 
each of whose indecomposable subsets satisfies this condition. 
\end{definition}

\begin{remark}\namedLabel{non-uniqueness of conjugating element}
The choice of $g$ such that $\defininghomo{X}{x} = c_g$ is not unique, 
but the conjugation biset only depends on $\defininghomo{X}{x}$ 
as a \emph{homomorphism}. 
This is the basic reason that functors out of $\subBurnside(G)$ are conjugation invariant. 
By definition, $c_g(h_2) = g^{-1}h_2g\in K_x$, for all $h_2 \in L_x \subset H_2$, 
where $g^{-1}L_xg=K_x\subset H_1$. 
We have  $c_{g} = c_{g_1}\colon L_x \to K_x$ if and only if
$g_1g^{-1} \in C_G(L_x)$, where $C_G(L_x)$ denotes the centralizer of $L_x$ in $G$. 
\end{remark}

The formula for the change of base-point in \namedRef{base point change} 
shows that the definition of conjugation biset does not depend on the choice of point 
$x$ used to compute
$\defininghomo{X}{x}$: the element giving the conjugation may change, 
but not the fact that it is given by some conjugation.  
In particular, any biset which is isomorphic to 
a conjugation biset is itself a conjugation biset.

By \namedRef{The Mackey double coset formula}, the composition of two conjugation bisets 
is again a conjugation biset. 
In addition, by \namedRef{standard representation for conjugate}, 
the involution $\congtran$ restricts to give an involution 
$\congtran\colon \subBurnside(G) \to \subBurnside(G)$.

\subsection{Induction and restriction}
There are two extreme cases of indecomposable bifree bisets 
$\biset{H_2}{X}{H_1}$.
\begin{enumerate}
\item We say $\biset{H_2}{X}{H_1}$ is a \emph{restriction} if the $H_1$ action is
transitive.
\item We say $\biset{H_2}{X}{H_1}$ is an \emph{induction} if the $H_2$ action is
transitive.
\end{enumerate}
We say that $\biset{H_2}{X}{H_1}$ is an \emph{isomorphism} if it is both
a restriction and an induction.

\begin{proposition}
Let  $\biset{H_2}{X}{H_1}$ be an indecomposable, bifree biset.
Pick a point $x\in X$ and let $[L_x, \defininghomo{X}{x}, K_x]$ be
the standard representation.
\begin{enumerate}
\item $\biset{H_2}{X}{H_1}$ is a restriction
if and only if $K_x = H_1$.
\item $\biset{H_2}{X}{H_1}$ is an induction if and only if $L_x = H_2$.
\end{enumerate}
\end{proposition}
\begin{proof}
Immediate from \namedRef{action orbits}.
\end{proof}

\begin{proposition}\label{conjugation_isomorphism}
If $\biset{H_2}{X}{H_1}$ is an isomorphism,
$\congtran(\biset{H_2}{X}{H_1})$ is the inverse isomorphism.
Conversely, if  $\biset{H_2}{X}{H_1}$ has an inverse then it is
an isomorphism.
This justifies the terminology.
\end{proposition}
\begin{proof}
Let $\biset{H_1}{Y}{H_2}$ be the inverse for  $\biset{H_2}{X}{H_1}$.
Then  $\biset{H_2}{X}{H_1} \times_{_{H_1}}  \biset{H_1}{Y}{H_2}$ is 
isomorphic to $\biset{H_2}{H_2}{H_2}$.
Since $( \biset{H_2}{H_2}{H_2})/H_2$ is a point, so is
$\bigl(\biset{H_2}{X}{H_1} \times_{_{H_1}}  \biset{H_1}{Y}{H_2}\bigr)/H_2 =
(X_1/H_1) \times (Y/H_2)$.
It follows that $Y$ is a restriction.
The other equation for the inverse shows that $Y$ is also an induction.

If $\biset{H_2}{X}{H_1}$ is both an induction and a restriction, it follows that the
standard representation is
$H_2$, $H_1$ and some isomorphism $\defininghomo{X}{x}$.
By \namedRef{standard representation for conjugate} the standard representation for
$\congtran(X)$ is $[H_1, \superdefininghomo{X}{x}{-1}, H_2]$.
By \namedRef{The Mackey double coset formula}, $\congtran(X)$ is the inverse for $X$.
\end{proof}

It further follows from \namedRef{The Mackey double coset formula} that the composition of
two restrictions is a restriction and the composition of two inductions is an induction.
More explicitly, we have the following.

\begin{proposition}
Let $[L_{x_1}, \defininghomo{X_1}{x_1}, K_{x_1}]$ be the standard representation for the
biset $\biset{H_2}{X_1}{H_1}$ at $x_1$ and let
$[L_{x_2}, \defininghomo{X_2}{x_2}, K_{x_2}]$ be the standard representation for the
biset $\biset{H_3}{X_2}{H_2}$ at $x_2$.
\begin{enumerate}
\item If $X_1$ and $X_2$ are both restrictions then so is $X_2\times_{_{H_2}} X_1$.
\item If $X_1$ and $X_2$ are both inductions then so is $X_2\times_{_{H_2}} X_1$.
\item If $X_1$ is an induction and if $X_2$ is a restriction then
$[L_{x_2},\defininghomo{X_1}{x_1} \circ \defininghomo{X_2}{x_2},K_{x_1}]$  is the
standard representation for the composition using the point $\quotientpoint{x_2}{x_1}$
\end{enumerate}
\end{proposition}
\begin{proof}
This follows from \namedRef{standard representation of a product} and some
standard set theory.
\end{proof}

\begin{definition}
We define three subcategories of $\Burnside(G)$, denoted $\Res_{\bA}(G)$, 
$\Ind_{\bA}(G)$ and $\Iso_{\bA}(G)$.
The objects of any of these categories are all the objects of $\Burnside(G)$.
The morphisms in $\Res_{\bA}(G)$ are the set of all restrictions, the morphisms
in $\Ind_{\bA}(G)$ are the set of all inductions and the morphisms in 
$\Iso_{\bA}(G)$ are the isomorphisms. \qed
\end{definition}

\begin{remark}
Because $\Hom_{\Burnside(G)}(H_1, H_2)$ is the free abelian group on the
indecomposable bifree $H_2$\,-$H_1$ bisets,
$\Hom_{\Res(G)}(H_1, H_2) \subset \Hom_{\Burnside(G)}(H_1, H_2)$, 
and indeed it is a summand.
There is a similar statement for each of $\Ind(G)$ and $\Iso(G)$.
Moreover, $\Iso(G) = \Res(G) \cap \Ind(G)$.
\end{remark}

Let $\biset{H_2}{X}{H_1}$ be an indecomposable bifree biset.
After choosing a point $x\in X$, we can display $X$ as a composition.

\begin{definition}\label{standardmaps}
Let $[L_x, \defininghomo{X}{x}, K_x]$ be the standard representation for $X$ at $x$.
We may regard  $H_2$ is an $H_2$\,-$L_x$ biset, and  $H_1$ as a 
$K_x$\,-$H_1$ biset, via group multiplication in $G$.
\begin{enumerate}
\item We define
$\standardInd{x} = \biset{H_2}{(H_2)}{L_x}$, note that it is an induction and
call it the \emph{standard induction for $X$ at $x$}.
\item We define
$\standardRes{x} = \biset{K_x}{(H_1)}{H_1}$, note that it is a restriction and
call it the \emph{standard restriction for $X$ at $x$}.
\item
There are two evident isomorphisms.
Make $K_x$ into a left $L_x$ set using $\defininghomo{X}{x}$,
and into a right $K_x$ set via group multiplication.
This makes $K_x$ into an $L_x$\,-$K_x$ biset.
Similarly make $L_x$ into an $L_x$\,-$K_x$ biset using 
$\superdefininghomo{X}{x}{-1}$.
Let  $\biset{L_x}{(K_x)}{K_x}$ be denoted $\standardLIso{x}$ and let
$\biset{L_x}{(L_x)}{K_x}$ be denoted $\standardRIso{x}$.
\end{enumerate}
\end{definition}
Note that $\standardRIso{x}$ and $\standardLIso{x}$ are isomorphic 
as bisets via the bijection $\defininghomo{X}{x}\colon L_x \to K_x$. The following composition formula was also observed in \cite{bouc0}*{Lemme 3} and \cite{bouc-thevenaz1}*{7.4}.

\begin{proposition}\label{factorize}
As $H_2$\,-$H_1$ bisets, an indecomposable bifree biset $X$ is isomorphic 
to the composition
$\standardInd{x} \circ\,  \standardRIso{x} \circ \standardRes{x}$.
\begin{enumerate}
\item The standard restriction $\standardRes{x}$ is always a conjugation restriction.
\item The standard induction $\standardInd{x}$ is always a conjugation induction.
\item
The indecomposable bifree biset $X$ is a conjugation biset if and only if
$\standardRIso{x}$ is a conjugation isomorphism.
\item
Moreover, $X$ is a restriction if and only if $\standardInd{x}$ is the identity;
$X$ is an induction if and only if $\standardRes{x}$ is the identity.
\end{enumerate}
\end{proposition}
\begin{proof}
The standard representation for $\standardInd{x}$ is 
$\biset{H_2}{[L_x, id, L_x]}{L_x}$,  so
it is clearly a conjugation biset.
The case  $\standardRes{x}$ is similar.
The standard representation for $\standardLIso{x}$ is 
$\biset{L_x}{[L_x, \defininghomo{X}{x}, K_x]}{K_x}$, 
and so $\standardLIso{x}$ is a conjugation biset if and only if $X$ is.

The function
$$\standardInd{x} \times_{L_x}\,  \standardRIso{x} \times_{K_x} \standardRes{x} \to X$$
which sends the image of $h_2\times \ell \times h_1$ to $h_2 \ell x h_1$ can be checked
to be a bijection of bisets.
The remarks about $X$ being a restriction or an induction are immediate.
\end{proof}

\begin{remark}\label{Vyx}
Of course the choice of $x\in X$ is not unique, so let $y = h_2 x h_1$.
Let $\standardRes{y}$, $\standardInd{y}$ and $\standardLIso{y}$ be the
corresponding bisets.
Define an $L_{y}$-${L_x}$ biset $ \changepointbiset{y}{x}$ to be $L_x$ 
with right multiplication by $L_x$ as the right action and use   
$c_{h_2}\colon L_y \to L_x$ to define the left action. 
Note that $ \changepointbiset{y}{x}$ is an isomorphism and a conjugation biset.
By \namedRef{base point change}, we have
$\standardInd{x}\cong\standardInd{y}\times_{_{L_y}}  \changepointbiset{y}{x}$.

Similarly, we define a $K_{y}$-$K_x$ biset $\changepointbisetK{y}{x}$ 
to be $K_x$ with right multiplication by $K_x$ and left multiplication 
defined by $c_{h_1}\colon K_y\to K_x$. 
We have 
$\standardRes{x}\cong \changepointbisetK{y}{x}^{-1} \times_{_{K_y}} \standardRes{y}$ 
and
$\standardRIso{x}\cong \changepointbiset{y}{x}^{-1}  \times_{_{L_y}} \standardRIso{y}  
\times_{_{K_y}}
\changepointbisetK{y}{x}$.
\end{remark}

\begin{definition}
We define three subcategories of $\subBurnsideDot(G)$, denoted
$\subInd(G)\subset \Ind_{\bA}(G)$, $\subRes(G)\subset \Res_{\bA}(G)$ 
and $\subIso(G)\subset\Iso_{\bA}(G)$. 
These are the subcategories with the same objects as the bigger categories, 
but whose morphisms
are all the morphisms which are conjugation bisets. \qed
\end{definition}

\section{Bifunctors into $\subBurnsideDot(G)$}\label{seven}
In this section we will define the bivariant functor 
$\DtoAdot\colon \cD(G)\to \subBurnsideDot(G)$ used in the statement of Theorem A.
Let $\cDdot(G)$ denote the category whose objects are pairs $(X, \mathbf b)$,
consisting of a finite $G$-space $X$ and an \emph{ordered} collection 
$\mathbf b = (b_1, \cdots, b_n)$ of base-points,
one for each $G$-orbit of $X$.
The morphisms are the $G$-maps (not necessarily base-point preserving).
There is a functor
$$\DdotD\colon \cDdot(G) \to \cD(G)$$
defined by forgetting the base-points.
Since every object of $\cD(G)$ is isomorphic to the image $\DdotD(X,\mathbf b)$ of an object of
$\cDdot(G)$, and $\DdotD$ induces a bijection on morphism sets, it follows that
$\DdotD$ gives an equivalence between the categories
$\cDdot(G)$ and $\cD(G)$, with inverse functor $\DtoDdot$
\cite{maclane2}*{IV.4, Theorem 1, p.~91}. 
Moreover, the inverse functor $\DtoDdot$ can be chosen so that we have the additivity formula
$$\DtoDdot(X\smdisjointunion Y) = \DtoDdot(X) \smdisjointunion \DtoDdot(Y)$$
for any finite $G$-sets $X$ and $Y$. 
This will be needed later to verify axiom (M2) for additive functors out of $\subBurnsideDot(G)$.

\medskip
We will now define the remaining functors in the following diagram:
\def\dup{3pt}\def\ddown{-3pt}
\eqncount
\begin{equation}
\vcenter{\xymatrix@C+4pt{%
&\Or(G)\ar@<\dup>[r]^{\upi}\ar@<\ddown>[r]_{\downi}\ar[d]^{\OrtoDdot}  &\subBurnside(G)\ar[d]^{u}&\cr
\cD(G)\ar[r]^{\DtoDdot}&\cDdot(G)\ar@<\dup>[r]^{\upj}\ar@<\ddown>[r]_{\downj}&
\subBurnsideDot(G) \cr}}
\end{equation}

The functor $\OrtoDdot\colon \Or(G) \to \cDdot(G)$ sends $H$ to the $1$-tuple $(G/H, eH)$ and a
  $G$-set map $G/H \to G/K$ to the same $G$-set map. 
In fact $\Or(G)$ as defined is isomorphic to a full subcategory of $\cDdot(G)$.

The functor $\downi$ is the identity on objects and sends a $G$-map $f\colon G/H \to G/K$
to the conjugation biset ${_{K}}K_{{g^{-1} H g}}$ where $f(e H) = g K$.  
This is well-defined, since a different choice $g k$, for $k \in K$, 
of representative yields an isomorphic biset.

Note that $1_{G/H}\colon G/H \to G/H$ goes to ${_{H}}H{_{H}}$ which is the identity.
Check that if
$f_1\colon G/H_1\to G/H_2$ and $f_2\colon G/H_2 \to G/H_3$ are $G$-maps and if
$f_1(e H_1) = g_1 H_2$ and $f_2(e H_2) = g_2 H_3$
then $f_2\circ f_1(e H_1) = (g_1 g_2) H_3$ and
$${_{H_3}}\bigl(H_3\bigr){_{g^{-1}_2 H_2 g_2}} \times_{H_2}
{_{H_2}}\bigl(H_2\bigr){_{g^{-1}_1 H_1 g_1}}$$
is isomorphic to ${_{H_3}}\bigl(H_3\bigr){_{(g_1 g_2)^{-1} H_1(g_1 g_2)}}$ 
by the map $(h_3, h_2) \mapsto h_3 g_2^{-1} h_2 g_2$.

The functor $\upi$ is also the identity on objects, but sends a 
$G$-map 
$f\colon G/H \to G/K$
to the conjugation biset ${_{g^{-1} H g}}K_{K}$ where $f(e H) = g K$.
Rather than check identity and composition, just note that
${_{g^{-1} H g}}K_{K}$ is isomorphic to 
$\congtran\bigl({_{K}}K_{{g^{-1} H g}}\bigr)$ by the
function which sends $k$ to $k^{-1}$,
so $\upi = \congtran \circ \downi$ and hence $\upi$ is a contravariant functor.

We define the functor $\downj$ on objects by additivity: every object of $\cDdot(G)$ has the form 
$$(X,\mathbf b)= \mathop{\disjointunion}_{i=1}^k \, (X_i,b_i),$$ where $(X_i,b_i) = G/H_i$ 
is an object of $\Or(G)$, and we send such an object to the ordered $n$-tuple
$(\downi(H_1), \dots, \downi(H_k))$. A morphism $\varphi\colon (X,\mathbf b) \to (Y,\mathbf c)$  in $\cDdot(G)$ is a collection $(\varphi_{i})$, $1\leqslant i \leqslant k$,  of $G$-maps of the form $\varphi_{i}\colon G/H_i \to G/K_{f(i)}$, where $f\colon \mathbf b\to \mathbf c$ is a function. Every morphism in $\subBurnsideDot(G)$ is represented a finite matrix of bifree conjugation bisets, where $\emptyset = 0$. We define $\downj(\varphi) = \alpha$ to be the morphism in $\subBurnsideDot(G)$ represented by the matrix $\alpha=  (\alpha_{ij})$, where $\alpha_{ij} = \downi(\varphi_{i})$, if $j = f(i)$, and $\alpha_{ij} =0$ otherwise. 

The functor $\upj$ is defined in a similar way.
Notice that  $\downj$ factors through the subcategory $\subInd(G)$, 
and $\upj$ factors through the subcategory $\subRes(G)$.

\begin{definition}\label{DtoBurnside}
We define the bivariant functor
$$\DtoAdot\colon \cD(G)\to \subBurnsideDot(G)$$
as the composition $\DtoAdot = (\upj, \downj)\circ \DtoDdot$. \qed
\end{definition}

\section{The proof of Theorem A}\label{eight}
We now state our main result, which is a more detailed version of Theorem A. 

\begin{theorem}\label{Mackeyrecognition}
If $F\colon \subBurnsideDot(G) \to \Ab$ is an additive functor, 
then $F\circ \DtoAdot\colon \cD(G) \to \Ab$ is 
a conjugation invariant Mackey functor. 
Conversely, any conjugation invariant Mackey functor factors 
uniquely through an additive functor out of $\subBurnsideDot(G)$.
\end{theorem}
\begin{proof} 
Suppose that $F\colon \subBurnsideDot(G)\to \Ab$ is an additive functor, 
and let $\cM = F\circ \DtoAdot\colon \cD(G) \to \Ab$ denote its composition 
with the bivariant functor $\DtoAdot$. 
Then $\cM$ is conjugation-invariant. 
The Mackey property (M2) is just additivity, so it remains to consider (M1). 
Let 
\eqncount
\begin{equation}\label{diagram:Mone}
{\vcenter
{\xymatrix
{S\ar[r]^\Psi\ar[d]_\Phi&S_2\ar[d]^\varphi\\S_1\ar[r]^\psi&T}}}
\end{equation}
be a pull-back diagram of finite $G$-sets. 
We first remark that a 
$G$-map $u \colon S \to S^\prime$ between finite $G$-sets is determined 
by its restriction to the disjoint $G$-orbits in $S$. 
In fact, each orbit in $S$ is mapped by $u$ into exactly one 
$G$-orbit of $S^\prime$. 
In particular, if $u \colon S \to S^\prime$ is an isomorphism 
of finite $G$-sets, then by 
Proposition \ref{conjugation_isomorphism} the induced maps 
$u_\cM$ and $u^\cM$ are both isomorphisms with $ (u_\cM)^{-1} = u^\cM$.

This remark implies that property (M1) depends only on 
the isomorphism class of diagram (\ref{diagram:Mone}). 
More precisely, let $c\colon S_1 \xrightarrow{\approx} S_1^\prime$ and 
$d\colon S_2  \xrightarrow{\approx}  S_2^\prime$ be $G$-isomorphisms, 
and consider the commutative diagram of $G$-sets and $G$-maps: 
$$
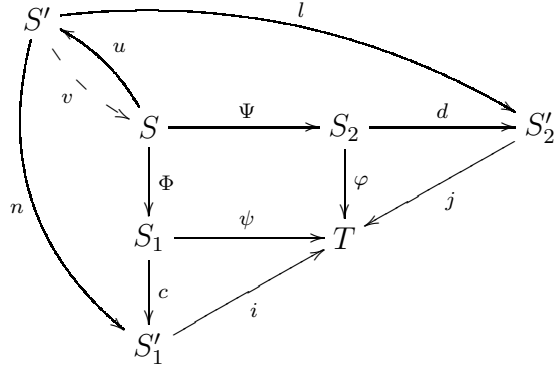
\xymatrix{S^\prime\ar@/^1.5pc/[drrrrr]^l\ar@/_2pc/[dddr]_n\ar@/_0.5pc/@{-->}[dr]_v&&&&&\cr
&S\ar[rr]^\Psi\ar[d]^\Phi\ar@/_0.5pc/[ul]_u&&S_2\ar[rr]^d\ar[d]^\varphi&&S_2^\prime\ar[dll]^j\cr
&S_1\ar[rr]^\psi\ar[d]^c&&T&\cr
&S_1^\prime\ar[urr]_i&&}$$
where $S^\prime$ is the pullback of $i$ and $j$, 
with $\psi = i \circ c$ and $\varphi = j \circ d$. 
There is an induced isomorphism $u \colon S \to S^\prime$ 
with inverse $v$, such that  $l \circ u = d\circ \Psi$, and $c\circ \Phi = n \circ u$.  
If we have the property (M1) for the pullback $(S^\prime, l, n)$, 
then the property (M1) holds for the pullback $(S, \Psi, \Phi)$. 
This is an easy diagram chase starting with the left-hand side of 
the required formula 
\eqncount
\begin{equation}\label{Mone}
\varphi^\cM\circ \psi_\cM = \Psi_\cM\circ \Phi^\cM 
\end{equation}
and substituting in the expressions above for $\psi$ and $\varphi$.

By the additivity property (M2), this formula
follows from the basic cases where $S_1$, $S_2$, and $T$ are transitive 
$G$-sets. 
Let $S_1 = G/H_1$, $S_2=G/H_2$ and $T=G/K$, 
and choose elements $g_1, g_2 \in G$ such that 
$g_i^{-1}H_ig_i \subseteq K$, for $i = 1,2$. 
Since $\cM$ is conjugation invariant, 
we may assume that $\psi(eH_1) = g_1K$ and $\varphi(eH_2) = g_2K$ 
(varying the choice of $g_1$, $g_2$ will not affect the maps induced by $\cM$). 

For $i=1$ and $i=2$ there are $G$-isomorphisms 
$\theta_i\colon G/H_i \to G/H_i^\prime$,  given by
$eH_i \mapsto g_iH_i^\prime$ where 
$H_i^\prime = g_i^{-1}H_i g_i$, with the property that  the pullback 
$G/H_1^\prime \subset G/K \supset G/H_2^\prime$ is isomorphic to 
$S = G/H_1 \times_{G/K} G/H_2$. 
It follows (by the remarks above) 
that it is enough to check formula
(\ref{Mone}) in the special case where $H_1$ and $H_2$ 
are actually subgroups of $K$. 
Hence, we may assume that $g_1 = g_2 = e$.

We will regard
$\biset{H_2}{K}{H_1}$
as an $H_2$-$H_1$ biset via the natural actions $H_2\subset K$ and $H_1\subset K$. 
The pullback $G$-set $$S = \bigl\{(x H_1, y H_2) \vv xK = y K\bigr\},$$ has the 
$G$-action defined by 
$g\cdot (x H_1, y H_2) = (g x H_1, g y H_2)$, for all $g\in G$. 
It follows that in each $G$-orbit there is always a representative of the form 
$(x H_1, e H_2)$ with $x\in K$ and
$(h_2 x h_1 H_1, e H_2) = (x H_1, e H_2)$ for all $h_2\in H_2$ and all $h_1\in H_1$.
In other words, the set of $G$-orbits in $S$ is in bijection with the quotient 
$H_2\backslash K / H_1$ of the biset $\biset{H_2}{K}{H_1}$. 

The isotropy subgroup
$$G_{(x H_1, e H_2)} = \bigl\{ g\in G \vv g x H_1 = x H_1,\  g H_2 = e H_2\bigr\} = 
H_2 \cap H_1^{x^{-1}} $$
where $H_1^{x^{-1}} = xH_1x^{-1}$ and $x \in K$ as above.
Therefore, we have a bijection of $G$-sets 
$$S = \bigsqcup_{x \in H_2\backslash K / H_1} G/ H_2 \cap H_1^{x^{-1}}\ .$$
In terms of biset morphisms, the composition
\eqncount
\begin{equation}\label{formula1}
\Psi_\cM\circ \Phi^\cM = 
\sum F\bigl (\biset{H_2}{H_2}{H_2\cap H_1^{x^{-1}}}\bigr ) \circ\,  
F\bigl (\biset{H_2\cap H_1^{x^{-1}}}{(H_2\cap H_1^{x^{-1}})}{H_1\cap H_2^{x}}\bigr )\circ 
F\bigl (\biset{H_1\cap H_2^{x}}{H_1}{H_1}\bigr )
\end{equation}
where the middle biset is just the conjugation $G$-isomorphism
$c_{x^{-1}}\colon H_1\cap H_2^{x} \to H_2\cap H_1^{x^{-1}}$.

On the other hand, the composition 
$$\varphi^\cM\circ \psi_\cM = F(\biset{H_2}{K}{H_1})\ .$$
To establish formula (\ref{Mone}), 
we consider the $H_2$\,-$H_1$ biset bijection
$$\biset{H_2}{K}{H_1} = \bigsqcup_{x \in H_2\backslash K / H_1} H_2\,x\,H_1 $$
from the double coset decompostion. 
The standard representation of the component 
$X = H_2\,x\,H_1$ expressed in the standard form is 
$$ H_2\,x\,H_1 = \biset{H_2}{[L_x, \defininghomo{X}{x}, K_x]}{H_1} =
[H_2\cap H_1^{x^{-1}},  c_{x},  H_1\cap H_2^{x}]~.$$
This follows from Remark \ref{define_parts}: 
$L_x = \bigl\{h_2\in H_2\vv h_2x = xh_1, \ h_1 \in H_1\bigr\} = H_2 \cap H_1^{x^{-1}}$, 
$K_x = \bigl\{h_1 \in H_1\vv xh_1 = h_2x, \ h_2 \in H_2 \bigr\} =  H_1 \cap H_2^{x}$ 
and 
$\defininghomo{X}{x} = c_x$.

Now by Proposition \ref{factorize}  we see that
the composition
$$\varphi^\cM\circ \psi_\cM = F(\biset{H_2}{K}{H_1})= 
\sum F(\standardInd{x}) \circ\,  F(\standardRIso{x})\circ F(\standardRes{x})\ .$$
But, by inspection, the right-hand side of this formula is just the expression for
$ \Psi_\cM\circ \Phi^\cM$ in formula (\ref{formula1}). 
This proves the property (M1).

\smallskip
Conversely, suppose that $\cM\colon \cD(G) \to \Ab$ is 
a conjugation-invariant Mackey functor. 
We define the functor $F$ on objects by setting $F(H) = \cM(G/H)$, 
and extending additively. 
Since the bivariant functor $\DtoAdot\colon \cD(G) \to \subBurnsideDot(G)$ 
is surjective on objects, this formula defines $F$ uniquely on objects.

The morphisms in $\subBurnsideDot(G)$ 
are finite matrices of bifree conjugation $H_2$\,-$H_1$ bisets. 
Any such biset is a disjoint union of indecomposable conjugation bisets $X$, 
uniquely up to ordering, and after picking a base point $x\in X$ 
we have the standard representation 
$X \cong \biset{H_2}{[L_x, \defininghomo{X}{x}, K_x]}{H_1}$ 
and the factorization 
$$\biset{H_2}{X}{H_1}\cong\biset{H_2}{[L_x, \defininghomo{X}{x}, K_x]}{H_1} \cong \standardInd{x} \circ\,  \standardRIso{x} \circ \standardRes{x}$$
of Proposition \ref{factorize}.  
By Definition \ref{standardmaps}, it follows that the morphisms in 
$\subBurnsideDot(G)$ are generated by compositions of 
the following three types of bifree conjugation bisets:
$$ \text{(i)\ } \standardInd{x} = \biset{H_2}{(H_2)}{L_x}, \quad
\text{(ii)\ } \standardRes{x} = \biset{K_x}{(H_1)}{H_1}, \quad \text{and}\quad
\text{(iii)\ } \standardRIso{x}=\biset{L_x}{(L_x)}{K_x}\ .$$
Recall from Remark \ref{define_parts} that 
$$L_x = \{h_2\in H_2\vv h_2x = xh_1, \ h_1 \in H_1\} \subset H_2$$ 
and $$K_x =\{h_1 \in H_1\vv xh_1 = h_2x, \ h_2 \in H_2 \}\subset H_1\ .$$
Since $X$ is a conjugation biset, the map
$\defininghomo{X}{x}\colon L_x \to K_x$ is given by 
$c_{g_x}$, where $g_{x}^{-1}L_x g_{x} = K_x$, so we have
$$\defininghomo{X}{x}(h_2) = c_{g_x}(h_2) = g_{x}^{-1}h_2g_{x} = h_1\ .$$
The element $g_x$ need not be unique 
(see \namedRef{non-uniqueness of conjugating element}).
The isomorphism $\superdefininghomo{X}{x}{-1}$ gives $L_x$ 
the right $K_x$-action on $L_x$ used to define the third biset. 

We will define $F$ on the basic morphisms of these three types, 
by associating to each of these
bisets a $G$-map, and then applying the Mackey functor $\cM$: 
\begin{itemize}
\item to $\standardInd{x}$ associate the $G$-map 
$\GsetInd{L_x}{H_2}\colon G/L_x \to G/H_2$, 
with $\GsetInd{L_x}{H_2}(e L_x) = e H_2$; 
\item to $\standardRes{x}$ associate the $G$-map 
$\GsetRes{H_1}{K_x}\colon G/K_x \to G/H_1$, 
with $\GsetRes{H_1}{K_x}(e K_x) = e H_1$;
\item to $\standardRIso{x}$ associate the maps of $G$-sets
$\GsetRIso{K_x}{L_x}\colon G/K_x \to G/L_x$,
with $\GsetRIso{K_x}{L_x}(e K_x) = g_{x}^{-1}L_x$ for the
various choices of $g_x$. 
\end{itemize}
Now define 
$$\text{(i)}\ F(\standardInd{x}) = \cM_\ast(\GsetInd{L_x}{H_2})
\hskip 25pt
\text{(ii)}\ F(\standardRes{x}) = \cM^\ast(\GsetRes{H_1}{K_x})
\hskip 25pt
\text{(iii)}\ F(\standardRIso{x}) = \cM_\ast(\GsetRIso{K_x}{L_x})$$

Item (iii) is well-defined by Lemma \ref{conginv}, 
since $\cM$ is conjugation invariant. 
\gdef\GsetRIsoA#1#2{\mathfrak L_{#1}^{#2}}
We will write $\GsetRIsoA{K_x}{L_x}$ for any of the
$\GsetRIso{K_x}{L_x}$.

\medskip

For any morphism $X$ in $\subBurnside(G)$ we define
$F(X)$ by choosing a base point $x\in X$, writing 
$X = \standardInd{x}\circ \standardRIso{x} \circ \standardRes{x}$, and defining 
$F(X) = F(\standardInd{x})\circ F(\standardRIso{x}) \circ F(\standardRes{x})$. 
The formulas in Remark \ref{Vyx} show that the definition of $F(X)$ is 
independent of the choice of base point.

The functor $F$ is defined on all the morphisms in 
$\subBurnsideDot(G)$ 
by the additive (matrix) extension of these formulas.
Any relation in the Grothendieck group 
$\Hom_{\subBurnside(G)}(H_1, H_2)$ leads to a isomorphism 
$X \cong Y$ of finite bifree $H_2$\,-$H_1$ bisets. 
Since both sides are canonically (up to ordering) expressed as 
a disjoint union of $(H_2 \times H_1^{op})$-orbits, it is clear that 
$F$ is well-defined.

Finally, we must check that $F$ is a functor. 
Since $F( _HH_H) = id$, 
it remains to check that compositions are preserved. 
Suppose that $X_1 \in \Hom_{\subBurnside(G)}(H_1, H_2)$ and 
$X_2 \in \Hom_{\subBurnside(G)}(H_2, H_3)$. 
By additivity, we may assume that $X_1$ and $X_2$ are indecomposable. 
We must check that
$$F(X_2\circ X_1) = F(X_2) \circ F(X_1)~.$$
Consider the left-hand side of the formula, 
where 
$F(X_2\circ X_1) = F(X_2\times_{H_2} X_1)$ by definition. 
Pick a base point $x_1\in X_1$, and $x_2\in X_2$.
By \namedRef{parameter}, the components of $X_3 = X_2\times_{H_2} X_1$
are indexed by elements $h_2\in H_2$ representing the double cosets 
$\doublecoset{K_{x_2}}{H_2}{L_{x_1}}$. 
Each such component contributes a summand
$F\bigl([L_{x_3}, \defininghomo{X_3}{x_3}, K_{x_3}]\bigr)$, 
where $x_3 = [x_2 h_2, x_1]$. 
By \namedRef{The Mackey double coset formula}, the standard representation 
at this base point is 
\newNumber{basic left hand side}
$$\bigl[L_{x_3}, \defininghomo{X_3}{x_3}, K_{x_3}\bigr]
=
\left [\superdefininghomo{X_2}{x_2}{-1}
\bigl(K_{x_2}\cap L^{h_2^{-1}}_{x_1}\bigr), 
\defininghomo{X_1}{x_1}\hskip -8pt \circ c_{h_2}\circ \defininghomo{X_2}{x_2}, 
\defininghomo{X_1}{x_1}\bigl(K^{h_2}_{x_2} \cap L_{x_1}\bigr)\right]
\leqno(\ref{basic left hand side})~.$$ 
The right-hand side of the formula is 
$$F(X_2) \circ F(X_1) = 
F(\standardInd{x_2})\circ F(\standardRIso{x_2}) \circ F(\standardRes{x_2})
\circ
F(\standardInd{x_1})\circ F(\standardRIso{x_1}) \circ F(\standardRes{x_1})~.$$

By the Mackey double coset formula for $\cM$, property (M1), 
$$F(\standardRes{x_2})
\circ
F(\standardInd{x_1}) = 
\sum_{h_2\in\doublecoset{K_{x_2}}{H_2}{L_{x_1}}} 
F(K_{x_2})
\circ F(K_{x_2} \cap L^{h_2^{-1}}_{x_1})
\circ F(L_{x_1})$$
where
\begin{itemize}
\item $K_{x_2}$ is a $(K_{x_2})$\,-\,$\bigl(K_{x_2} \cap L^{h_2^{-1}}_{x_1}\bigr)$ 
biset via the two inclusions, and is a standard induction, 
\vskip 2pt
\item $K_{x_2} \cap L^{h_2^{-1}}_{x_1}$ is a 
$\bigl(K_{x_2} \cap L^{h_2^{-1}}_{x_1}\bigr)$\,-\,$\bigl(L_{x_1} \cap K^{h_2}_{x_2}\bigr)$ 
biset via the evident inclusion on the left and  conjugation by 
$h_2^{-1}$ on the right, and is a standard conjugation, 
\vskip 2pt
\item $L_{x_1}$ is an 
$\bigl(L_{x_1} \cap K^{h_2}_{x_2}\bigr)$\,-\,$(L_{x_1})$ biset 
via the two inclusions and is a standard restriction. 
\end{itemize}

Hence $F(X_2) \circ F(X_1) = 
\sum_{h_2\in\doublecoset{K_{x_2}}{H_2}{L_{x_1}}} 
A[h_2]\circ B[h_2] \circ C[h_2]$.
where
$$\begin{matrix}
A[h_2] &=& \cmmA{\GsetInd{L_{x_2}}{H_3}}\circ 
\cmmA{\GsetRIsoA{K_{x_2}}{L_{x_2}}} \circ 
\cmmA{\GsetInd{\Ax}{K_{x_2}}}\hfill\\\noalign{\vskip 8pt}
B[h_2] &=& \cmmA{\GsetRIsoA{\Bx}{\Ax}}\hfill\\\noalign{\vskip 8pt}
C[h_2] &=& \cmmB{\GsetRes{L_{x_1}}{\Bx}}\circ 
\cmmA{\GsetRIsoA{K_{x_1}}{L_{x_1}}} \circ 
\cmmB{\GsetRes{H_1}{K_{x_1}}}\hfill\\
\end{matrix}$$

To further analyze $A[h_2]$ consider the commutative diagram
$$\begin{matrix}
\hbox to 0pt{\hss$H_3\hskip 9pt\supset\hskip 10pt$} L_{x_2}&
\xleftarrow{\superdefininghomo{X_2}{x_2}{-1}}&
K_{x_2}\\\noalign{\vskip 4pt}
\cup&&\cup\\
\superdefininghomo{X_2}{x_2}{-1}\bigl(\Ax)&
\xleftarrow{\superdefininghomo{X_2}{x_2}{-1}}{}&
\Ax\\
\end{matrix}
$$
which implies 
$$A[h_2] = \cmmA{
\GsetInd{\superdefininghomo{X_2}{x_2}{-1}(\Ax)}{H_3}} \circ
\cmmA{\GsetRIsoA{\Ax}{\superdefininghomo{X_2}{x_2}{-1}(\Ax)}}
~.$$

To analyze $C[h_2]$ first note that 
$\cmmA{\GsetRIsoA{K_{x_1}}{L_{x_1}}} = 
\cmmB{\GsetRIsoA{L_{x_1}}{K_{x_1}}}$.
Then consider the commutative diagram
$$\begin{matrix}
L_{x_1}&\xrightarrow{\defininghomo{X_1}{x_1}}&
K_{x_1}\hbox to 0pt{\hskip 10pt$\subset\hskip 8pt H_1$\hss}\\\noalign{\vskip 4pt}
\cup&&\cup\\
\Bx&
\xrightarrow{\defininghomo{X_1}{x_1}}{}&
\defininghomo{X_1}{x_1}(\Bx)\\
\end{matrix}
$$
which implies
$$\begin{matrix}
C[h_2] &= &\cmmB{\GsetRes{L_{x_1}}{\Bx}}
\circ
\cmmB{\GsetRIsoA{L_{x_1}}{K_{x_1}}} \circ 
\cmmB{\GsetRes{H_1}{K_{x_1}}}\\\noalign{\vskip 8pt}& = &
\cmmB{\GsetRIsoA{\Bx}{\defininghomo{X_1}{x_1}(\Bx)}}\circ
\cmmB{\GsetRes{H_1}{\defininghomo{X_1}{x_1}(\Bx)}}\\
\end{matrix}
$$
\medskip
To analyze the conjugations, note that after substituting 
the new expressions for $A[h_2]$ and $C[h_2]$ we have three conjugations now 
occurring together in $A[h_2]\circ B[h_2]\circ C[h_2]$: 
$$\begin{matrix}
\cmmA{\GsetRIsoA{\Ax}{\superdefininghomo{X_2}{x_2}{-1}(\Ax)}} \circ
\cmmA{\GsetRIsoA{\Bx}{\Ax}}\circ 
\cmmB{\GsetRIsoA{\Bx}{\defininghomo{X_1}{x_1}(\Bx)}} & = \\
\cmmA{\GsetRIsoA{\Ax}{\superdefininghomo{X_2}{x_2}{-1}(\Ax)}} \circ
\cmmA{\GsetRIsoA{\Bx}{\Ax}}\circ 
\cmmA{\GsetRIsoA{\defininghomo{X_1}{x_1}(\Bx)}{\Bx}} & = \\
\cmmA{\GsetRIsoA{\defininghomo{X_1}{x_1}(\Bx)}{\superdefininghomo{X_2}{x_2}{-1}(\Ax)}}\hfill\\
\end{matrix}
$$

Hence 
$$\begin{matrix}
A[h_2]\circ B[h_2]\circ C[h_2] =\hfill \\\noalign{\vskip 8pt}
\cmmA{
\GsetInd{\superdefininghomo{X_2}{x_2}{-1}(\Ax)}{H_3}}
\circ
\cmmA{\GsetRIsoA{\defininghomo{X_1}{x_1}(\Bx)}
{\superdefininghomo{X_2}{x_2}{-1}(\Ax)}}
\circ
\cmmB{\GsetRes{H_1}{\defininghomo{X_1}{x_1}(\Bx)}}\\
\end{matrix}
$$
It now follows from (\ref{basic left  hand side}) that
$$A[h_2]\circ B[h_2]\circ C[h_2] = F(\standardInd{x_3})\circ 
F( \standardRIso{x_3}) \circ F(\standardRes{x_3})$$
for each component $x_3 = [x_2h_2,x_1]$. 
Therefore $F(X_2\circ X_1) = F(X_2)\circ F(X_1)$.
\end{proof}

\providecommand{\bysame}{\leavevmode\hbox to3em{\hrulefill}\thinspace}
\providecommand{\MR}{\relax\ifhmode\unskip\space\fi MR }
\providecommand{\MRhref}[2]{%
\href{http://www.ams.org/mathscinet-getitem?mr=#1}{#2}
}
\providecommand{\href}[2]{#2}
\renewcommand{\MR}[1]{}

\end{document}